\pdfoutput=1
\RequirePackage{ifpdf}
\ifpdf 
\documentclass[pdftex]{sigma}
\else
\documentclass{sigma}
\fi

\numberwithin{equation}{section}

\newtheorem{Theorem}{Theorem}[section]
\newtheorem{Corollary}[Theorem]{Corollary}
\newtheorem{Conjecture}[Theorem]{Conjecture}
\newtheorem{Lemma}[Theorem]{Lemma}
\newtheorem{Proposition}[Theorem]{Proposition}
\newtheorem{Question}[Theorem]{Question}

\newtheorem{problem}[Theorem]{Problem}

{ \theoremstyle{definition}
\newtheorem{Definition}[Theorem]{Definition}
\newtheorem{Example}[Theorem]{Example}
\newtheorem{Remark}[Theorem]{Remark}
\newtheorem{observation}[Theorem]{Observation}
}

\usepackage{tikz}
\usetikzlibrary{arrows}
\usetikzlibrary{matrix}
\usetikzlibrary{positioning}
\usetikzlibrary{snakes}
\usetikzlibrary{calc}

\usepackage{booktabs}
\usepackage{mathdots}

\usepackage{algorithm2e}

\DeclareMathOperator{\Tr}{Tr}
\DeclareMathOperator{\sgn}{sgn}
\DeclareMathOperator{\diag}{diag}
\newcommand{\bfy}{\mathbf{y}}

\newcommand{\ZZ}{\mathbb{Z}}

\newcommand{\CC}{\mathbb{C}}
\newcommand{\ot}{\otimes}
\newcommand{\SG}{\mathfrak{S}}

\newsavebox\MBox
\newcommand\Cline[2][red]{{\sbox\MBox{$#2$}%
		\rlap{\usebox\MBox}\color{#1}\rule[-1.2\dp\MBox]{\wd\MBox}{0.5pt}}}

\DeclareMathOperator{\uMPS}{uMPS}
\DeclareMathOperator{\Cyc}{Cyc}

\DeclareMathOperator{\Dih}{Dih}

\def\picuMPS{\tikz[baseline=0ex,scale=2]{
		\foreach \x in {0,1,2,4}
		\filldraw[black] (\x+0.25,0.25) circle (2pt);
		\foreach \x in {0,1,2,4}
		\foreach \y in {1}
		\draw (\x+0.25,\y-0.75) -- (\x+0.25,\y);
		\draw (0.25,0.25) -- (3,0.25);
		\draw (3.5,0.25) -- (4.25,0.25);
		\node at (3.25, 0.25)    {$\cdots$};
		\foreach \x in {1,2,3,5}
		\node at (\x - 0.9, 1)  {$n$};
		\foreach \x in {1,2,3}
		\node at (\x - 0.25, 0.4)  {$m$};
		\node at (3.75, 0.4)  {$m$};
		\draw    (0.25,0.25) to[out=180,in=180] (0.25,-0.5);
		\draw    (4.25,0.25) to[out=0,in=0] (4.25,-0.5);
		\draw    (0.25,-0.5) to (4.25,-0.5);
		\node at (2.25, -0.05)   {$d$ sites};
		\node at (2.25, -0.35)  {$m$};
	}
}

\def\picOneTensor{\tikz[baseline=4ex,scale=1.5]{
		\foreach \x in {1}
		\filldraw[black] (\x+0.25,0.25) circle (2pt);
		\foreach \x in {1}
		\foreach \y in {1}
		\draw (\x+0.25,\y-0.75) -- (\x+0.25,\y);
		\draw (0.5,0.25) -- (2,0.25);
		\foreach \x in {2}
		\node at (\x - 0.9, 1)  {$n$};
		\foreach \x in {1,2}
		\node at (\x - 0.25, 0.4)  {$m$};
	}
}

\begin{document}
\allowdisplaybreaks

\newcommand{\arXivNumber}{2204.10363}

\renewcommand{\PaperNumber}{099}

\FirstPageHeading

\ShortArticleName{The Linear Span of Uniform Matrix Product States}

\ArticleName{The Linear Span of Uniform Matrix Product States}

\Author{Claudia DE LAZZARI~$^{\rm a}$, Harshit J.~MOTWANI~$^{\rm b}$ and Tim SEYNNAEVE~$^{\rm c}$}

\AuthorNameForHeading{C.~De~Lazzari, H.J.~Motwani and T.~Seynnaeve}

\Address{$^{\rm a)}$~Dipartimento di Matematica, Universit\`a di Trento,\\
\hphantom{$^{\rm a)}$}~Via Sommarive 14, 38123 Povo (TN), Italy}
\EmailD{\href{mailto:claudia.delazzari@unitn.it}{claudia.delazzari@unitn.it}}

\Address{$^{\rm b)}$~Department of Mathematics: Algebra and Geometry, Ghent University, 9000 Gent, Belgium}
\EmailD{\href{mailto:harshitjitendra.motwani@ugent.be}{harshitjitendra.motwani@ugent.be}}

\Address{$^{\rm c)}$~Department of Computer Science, KU Leuven, Celestijnenlaan 200A, 3001 Leuven, Belgium}
\EmailD{\href{mailto:tim.seynnaeve@kuleuven.be}{tim.seynnaeve@kuleuven.be}}

\ArticleDates{Received June 03, 2022, in final form December 15, 2022; Published online December 21, 2022}

\Abstract{The variety of uniform matrix product states arises both in algebraic geometry as a natural generalization of the Veronese variety, and in quantum many-body physics as a model for a translation-invariant system of sites placed on a ring. Using methods from linear algebra, representation theory, and invariant theory of matrices, we study the linear span of this variety.}


\Classification{15A69; 20G05; 81P45}

\section{Introduction}

Tensor networks are a powerful tool for the description of tensors living in a high dimensional space. Since their original conception in quantum many-body physics  \cite{AffKenLieTas:ValenceBondGroundStates,fannes1992finitely,OstRom:ThermoLimitDensityMatrixRenorm}, they have found a wide range of applications in different fields, such as numerical tensor calculus \cite{bachmayr2016tensor,cichocki2016tensor,hackbuschTensorSpaces, oseledets2011tensor}, algebraic complexity theory \cite{DvMaPeYe:Multi_Branching_Programs}, %
graphical models \cite{robevaSeigal}, and machine learning \cite{chen2018equivalence, cichocki2016tensor}.
In the study of these varieties of tensors, some important general results and developments have been achieved using methods from differential and complex geometry \cite{haegeman2014geometry}, algebraic geometry and representation theory \cite{barthel2022closedness,bernardi2022dimension,BucBucMic:HackbushConjecture,ChrGesStWer:Optimization,ChrLucVraWer:TensorNetworkRepGeom,hackl2020geometry, landsbergQiYe}.

In this article, we focus on a particular type of tensor networks: uniform matrix product states. From a quantum mechanics perspective, they model translation-invariant physical systems of sites placed on a ring. The geometry of uniform matrix product states has been extensively studied \cite{Critch2014AlgebraicGO,Czaplinski2019UniformMP,haegeman2014geometry, PerezGarciaVerstraeteMPS}, but our understanding of them is still far from complete, and several fundamental mathematical problems remain open.

Our geometric point of view is the following: we fix a tensor space $(\CC^n)^{\ot d}$, and consider the set of all tensors in this space that admit a translation-invariant matrix product state representation, with a given bond dimension $m$. After taking the closure of this set, we obtain an algebraic variety, which we denote by $\uMPS(m,n,d)$.
The ultimate goal would be to obtain a complete description of this variety, i.e., to find all defining equations. Phrased in this generality, this question is likely to be intractable. Indeed, even just determining which \emph{linear} equations, if any, vanish on our variety is poorly understood. This is precisely the goal of this paper:

\begin{problem} \label{problem: linSpanUMPS}
Describe the linear span $\langle\uMPS(m,n,d)\rangle$ of the variety $\uMPS(m,n,d)$. More precisely:
\begin{itemize}\itemsep=0pt
    \item What is the dimension of $\langle\uMPS(m,n,d)\rangle$? {\rm (}Question~{\rm \ref{question:main})}
     \item For which parameters $m,n,d \in \mathbb{N}$ does $\langle\uMPS(m,n,d)\rangle$ fill the ambient space? {\rm (}Question~{\rm \ref{question:fills})}
    \item How does $\langle\uMPS(m,n,d)\rangle$ decompose as a ${\rm GL}_n$-representation? {\rm (}Question~{\rm \ref{question:repPrecise})}
\end{itemize}
\end{problem}

Note that since taking the Zariski closure of any set does not change its linear span, the space $\langle\uMPS(m,n,d)\rangle$ is the linear span of all tensors that admit a $\uMPS$ representation with bond dimension $m$.

The variety $\uMPS(m,n,d)$ does not only arise from tensor networks, it is also a very natural geometric construction in its own right. Indeed, as we will soon see, $\uMPS(m,n,d)$ is the closed image of the polynomial map that takes as input an $n$-tuple of $m \times m$ matrices, and outputs the traces of all $d$-fold products of the given matrices. In this way, it is a natural generalization of the classical Veronese variety. Our main Problem \ref{problem: linSpanUMPS} is therefore equivalent to the following:
\begin{problem}
Let $A_0,\dots,A_{n-1}$ be $m \times m$ matrices with generic\footnote{More formally: the entries of the $A_i$ are the $nm^2$ variables of a multivariate polynomial ring over $\CC$. See also Definition~\ref{def:TraceAlgebra}.} entries. Which linear relations hold among the polynomials
\[
\Tr(A_{i_1}\cdots A_{i_d}),
\]
where $(i_1,\dots,i_d) \in [n]^d$?
\end{problem}
The ring generated by all polynomials $\Tr(A_{i_1}\cdots A_{i_d})$, where the generic matrices are fixed but we allow $d$ to vary, is known as the \emph{trace algebra}. In his classical work \cite{procesi1976invariant}, Procesi described how to obtain all relations between the generators of this ring in principle. Note however, that the question we are asking is slightly different: we are only interested in relations between traces of matrix products of a \emph{fixed length} $d$.

This article is divided into three sections. In Section~\ref{sec:preliminaries}, we define the variety of uniform matrix product states, and describe its natural symmetries.
In Section~\ref{sec:comp}, we undertake a~computational study of the space $\langle\uMPS(m,n,d)\rangle$ in the smallest nontrivial case ${m=n=2}$. We describe an algorithm that can compute this space, viewed as a ${\rm GL}_2$-representation: more precisely, we exploit and compute its decomposition into ${\rm GL}_2$-submodules. Based on this technique, we obtain a conjectured formula for the character and, in particular, the dimension of the weight spaces, and consequently of the whole space $\langle\uMPS(2,2,d)\rangle$.
Section~\ref{sec:results} contains our main theoretical results.
In Section~\ref{sec:CH}, we introduce a powerful method to find linear equations that vanish on $\langle\uMPS(m,n,d)\rangle$, based on the Cayley--Hamilton theorem:
	\theoremstyle{plain}
	\newtheorem*{thm:linRel}{Theorem \ref{thm:linRel}}
	\begin{thm:linRel}
		Let $A_0,\dots,A_m,B$ be $m \times m$ matrices. Then for every $\ell \in {\mathbb N}$ it holds that
		\begin{equation*}
			\sum_{\sigma \in \mathfrak{S}_m, \tau \in C_{m+1}} \sgn(\sigma)\sgn(\tau)\Tr\big(A_{\tau(0)}B^{\sigma(0)}A_{\tau(1)}B^{\sigma(1)} \cdots A_{\tau(m-1)}B^{\sigma(m-1)}A_{\tau(m)}B^{\ell}\big) =0.
		\end{equation*}
		Here $\mathfrak{S}_m$ is the symmetric group acting on $\{0,1,\dots,m-1\}$, and $C_{m+1}$ is the cyclic group acting on $\{0,1,\dots,m\}$.
	\end{thm:linRel}
	Theorem \ref{thm:linRel} actually gives nontrivial trace relations, i.e., linear relations that do not follow either from cyclic permutations or reflections of the factors.
	As a corollary, we show that, already for $d$ at least quadratic in the bond dimension $m$, the linear span $\langle\uMPS(m,n,d)\rangle$ does not fill its natural ambient space $\Cyc^d(\CC^n)$, the space of cyclically symmetric tensors; significantly improving the state of the art.
	More precisely, the linear span of the space of uniform matrix product states is a proper subspace of the space of cyclically invariant tensors %
	under the following conditions: %
	\theoremstyle{plain}
	\newtheorem*{cor: lin span}{Corollary \ref{cor: lin span}}
	\begin{cor: lin span}
		If $n \geq 3$ and $d \geq \frac{(m+1)(m+2)}{2}$, then ${\uMPS(m,n,d)}$ is contained in a proper linear subspace of the space of cyclically invariant tensors.
	\end{cor: lin span}

In Section~\ref{sec:lineq22d} we study the special case $m=n=2$, based on the computations in Section~\ref{sec:CH}, and, using the trace parametrization, we provide an upper bound on the dimension of $\langle\uMPS(2,2,d)\rangle$ which is close to optimal. Moreover, we take some first steps towards proving our conjectured character formulas: in the first cases, i.e., for the first few weights, we can provide an explicit basis of generators for the weight spaces of the representation, again using our Cayley--Hamilton technique.

\section{Preliminaries} \label{sec:preliminaries}

We once and for all fix three parameters $m,n,d \!\in\! {\mathbb N}$. We will consider tensors in the space~$(\CC^n)^{\ot d}$. The standard basis of $\CC^n$ will be written as $\{e_0,\dots, e_{n-1}\}$. We abbreviate the set $\{0,\dots, n-1\}$ to $[n]$.

\begin{Definition}\label{def: umps}
		The uniform Matrix Product State parametrization is given by the map
		\begin{align}\label{eq:umpsMap}
		\begin{split}
			\phi\colon \  ({\mathbb C}^{m\times m})^n &\to ({\mathbb C}^n)^{\otimes d}, \\
			(A_0 ,\dots, A_{n-1})&\mapsto \sum_{0 \leq i_1,\dots, i_d \leq n-1} \Tr(A_{i_1}\cdots A_{i_d}) \ e_{i_1}\otimes \cdots \otimes e_{i_d}.
			\end{split}
		\end{align}
	\end{Definition}
We denote the image of this map by $\uMPS^{\circ}(m,n,d)$.
The closure of $\uMPS^{\circ}(m,n,d)$ in the Euclidean or equivalently Zariski topology over the complex numbers, is the algebraic variety of uniform matrix product states, denoted by ${\uMPS(m,n,d)}$.

\begin{Remark}
    If we think of $(\CC^{m\times m})^n$ as the space of $m \times m \times n$ tensors, then the $\uMPS$ parametrization takes a tensor in this space and contracts it $d$ times with itself in a circle; see Figure~\ref{fig:uMPS1}.
\end{Remark}

\begin{figure}[h!]
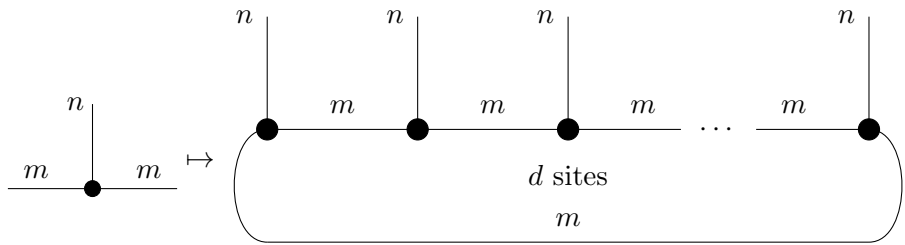

    \centering
    $$
    \picOneTensor \mapsto \picuMPS
    $$
    \caption{Graphical representation of map (\ref{eq:umpsMap}) defining $\uMPS(m,n,d)$. There are in total $d$ tensors of order $m\times m \times n$, contracted along the respective edges. The parameters $m$ and $n$ are called bond and physical dimension, respectively. The bond dimensions associated to the edges of the graph represent the dimension of the vector spaces along which the tensor contraction takes place. }%
    \label{fig:uMPS1}
\end{figure}

\begin{Remark}
An alternative name for $\uMPS$ is \emph{translation invariant matrix product states with periodic boundary conditions} \cite{Critch2014AlgebraicGO, PerezGarciaVerstraeteMPS}. The name ``uniform matrix product states'' is sometimes reserved for the thermodynamic limit, where the number $d$ of sites approaches infinity. Our terminology is consistent with \cite{Czaplinski2019UniformMP, haegeman2014geometry}.
\end{Remark}

As mentioned in the introduction, the main question we will try to answer in this article is the following:
\begin{Question} \label{question:main}
Determine the linear span of $\uMPS(m,n,d)$; i.e., the smallest vector subspace of $(\CC^n)^{\ot d}$ containing $\uMPS(m,n,d)$. In particular: what is the dimension of this space?
\end{Question}

\subsection{Cyclic and symmetric invariance}
The space $(\CC^n)^{\ot d}$ comes equipped with an action of the symmetric group $\SG_d$: for $\sigma \in \SG_d$ and $\omega = v_1 \ot \cdots \ot v_d \in (\CC^n)^{\ot d}$ we have
\[
\sigma \cdot (v_1 \ot \cdots \ot v_d) = v_{\sigma^{-1}(1)} \ot \cdots \ot v_{\sigma^{-1}(d)}.
\]
The symmetric group $\SG_d$ naturally contains the cyclic group $C_d$ and the dihedral group $D_{2d}$ as subgroups. To be precise: let $r,s \in \SG_d$ be the cyclic permutation and reflection defined respectively by
\[
r(i) = i+1 (\operatorname{mod} d) \qquad \text{and} \qquad s(i) = d+1-i,
\]
then $C_d \subseteq \SG_d$ is the cyclic subgroup generated by $r$, and $D_{2d}$ is the subgroup generated by $r$ and $s$.
The \emph{cyclically symmetric tensors} and \emph{dihedrally symmetric tensors} are then the elements of $(\CC^n)^{\ot d}$ that are invariant under the action of these subgroups:
    \begin{gather*}
        \Cyc^d({\mathbb C}^n) := \big\{ \omega \in ({\mathbb C}^n)^{\ot d} \mid \sigma \cdot \omega = \omega \ \forall \sigma \in C_d\big\},\\
        \Dih^d({\mathbb C}^n) := \big\{ \omega \in ({\mathbb C}^n)^{\ot d} \mid \sigma \cdot \omega = \omega \  \forall \sigma \in D_{2d}\big\}.
    \end{gather*}
Note that both are linear subspaces of $(\CC^n)^{\ot d}$, and that $\Dih^d({\mathbb C}^n) \subseteq \Cyc^d({\mathbb C}^n)$, where the inclusion is strict unless $d \leq 2$ or $n=2$ and $d \leq 5$.

\begin{observation}
The set $\uMPS(m,n,d)$ is a subset of the space of cyclically invariant tensors $\Cyc^d({\mathbb C}^n)\subset ({\mathbb C}^n)^{\otimes d}$ because of the trace invariance under cyclic permutations of the matrices: given $M_1,\dots,M_{d}\in {\mathbb C}^{m \times m}$ then
\[
\Tr(M_1\cdots M_{d})= \Tr(M_{\sigma(1)}\cdots M_{\sigma(d)}),
\]
for $\sigma \in C_d$.
\end{observation}
In other words, we can think of $\uMPS(m,n,d)$ as a subvariety of the ambient space $\Cyc^d({\mathbb C}^n)$. As noted in \cite[Corollary 3.18]{Czaplinski2019UniformMP}, if we fix $n$ and $d$ and let $m$ grow, the space $\uMPS(m,n,d)$ will eventually fill this entire ambient space.
\begin{Question} \label{question:fills}
For fixed $n$ and $d$, what is the smallest $m$ such that $\langle \uMPS(m,n,d) \rangle \!=\! \Cyc^d(\CC^m)$? \end{Question}
It is known that for $m=d$, equality holds \cite[Proposition 3.11]{Czaplinski2019UniformMP}. On the other hand, it follows from a dimension count (\cite[Theorem 3.14]{Czaplinski2019UniformMP}, see also \cite{navascues2018bond}) that if $d \gg m$, the inclusion $\langle \uMPS(m,n,d) \rangle \subset \Cyc^d(\CC^n)$ is strict. In Section~\ref{sec:CH} we will prove that already for  $d=O\big(m^2\big)$, we have a strict inclusion.
\begin{observation}\label{rmk:reflection}
    In the case $m=n=2$, we have the stronger inclusion
\[
    \uMPS(2,2,d) \subseteq \Dih^d(\CC^n).
\]
    This is a consequence of the identity $\Tr(A_{i_1}\cdots A_{i_{d}})= \Tr(A_{i_{d}}\cdots A_{i_1})$, which holds for any pair of $2\times 2$ matrices $A_0$, $A_1$ and sequence $i_1,\dots,i_{d}$ with $i_j \in \{0,1\}$. See \cite[Theorem 1.1]{greene2014traces}
    \end{observation}
\begin{Example}
    We consider $\uMPS(2,2,6)$.
    For every $A_0,A_1\in {\mathbb C}^{2\times 2}$, we have the trace relation
    $\Tr\big(A_0^2A_1^2A_0A_1\big)=\Tr\big(A_1A_0A_1^2A_0^2\big)$, that does not come from invariance of trace under cyclic permutations of the matrices.
\end{Example}

\subsection[GL_n-invariance]{$\boldsymbol{{\rm GL}_n}$-invariance}
The general linear group ${\rm GL}_n$ naturally acts on the space $({\mathbb C}^n)^{\otimes d}$: given $g\in {\rm GL}_n$ and $\omega=v_1 \otimes \cdots \otimes v_d\in ({\mathbb C}^n)^{\otimes d}$, we have
 \begin{equation*}%
 g\cdot (v_1\otimes \cdots \otimes v_N)= (g \cdot v_1)\otimes \cdots \otimes (g\cdot v_N).
 \end{equation*}
 Clearly, $\Cyc^d({\mathbb C}^n)$ and $\Dih^d({\mathbb C}^n)$ are invariant under this action. The following computation shows that $\uMPS(m,n,d)$ is invariant under this action as well:
 \begin{align*}
 	g\cdot \phi(A_1,\dots ,A_n)
 	&=\sum_{j_1,\dots, j_d=1}^{n} \Tr(A_{j_1}\cdots A_{j_d})  \bigg(\sum_{i_1=1}^{n}{g_{i_1,j_1}e_{i_1}}\bigg)\otimes \cdots \otimes \bigg(\sum_{i_d=1}^{n}{g_{i_d,j_d}e_{i_d}}\bigg) \\
 	&=\sum_{j_1,\dots, j_d=1}^{n} \sum_{i_1,\dots,i_d=1}^n g_{i_1,j_1} \cdots g_{i_d,j_d}\Tr(A_{j_1}\cdots A_{j_d}) \ e_{i_1}\otimes \cdots \otimes e_{i_d} \nonumber \\
 	&=\sum_{j_1,\dots, j_d=1}^{n} \sum_{i_1,\dots,i_d=1}^n \Tr(g_{i_1,j_1} A_{j_1}\cdots g_{i_d,j_d} A_{j_d}) \ e_{i_1}\otimes \cdots \otimes e_{i_d}\nonumber \\
 	&=\sum_{i_1,\dots,i_d=1}^n\Tr\bigg[\bigg(\sum_{j_1=1}^n g_{i_1,j_1}A_{j_1}\Big)\cdots \Big(\sum_{j_d=1}^n g_{i_d,j_d}A_{j_d}\bigg)\bigg]e_{i_1}\otimes \cdots \otimes e_{i_d}\nonumber \\
 	&= \phi\bigg(\sum_{j=1}^n g_{1,j} A_j, \dots ,\sum_{j=1}^n g_{n,j} A_j\bigg). \nonumber
 \end{align*}

 This means that the space $\langle \uMPS(m,n,d) \rangle$ we are interested in is naturally a representation of~${\rm GL}_n$.
 In the remainder of this subsection, we briefly recall what we will use about representation theory of ${\rm GL}_n$, and fix notation.

We once and for all fix a torus $T \subset {\rm GL}_n$, consisting of all diagonal matrices; and identify~$T$ with $(\CC^*)^n$.
For $\lambda=(\lambda_0,\dots,\lambda_{n-1}) \in \ZZ^n$ and $t=\diag(t_0,\dots,t_{n-1}) \in T$, we write $t^\lambda = t_0^{\lambda_0} \cdots t_{n-1}^{\lambda_{n-1}} \in \CC^*$.
Let $V$ be any representation of ${\rm GL}_n$.
For any $\lambda \in \ZZ^n$, the \emph{weight space} $V_{\lambda}$ is defined as
\[
V_{\lambda} = \big\{v \in V \mid t\cdot v = t^\lambda v \  \forall t \in T\big\}.
\]
It is a well-known fact from representation theory that
\[
V = \bigoplus_{\lambda \in \ZZ^n}{V_{\lambda}}
\]
as vector spaces; and that knowing the dimensions of the weight spaces uniquely determines the representation $V$ up to isomorphism.
The polynomial
\[
\chi_V = \sum_{\lambda \in \ZZ^n}{(\dim V_{\lambda}) t^{\lambda}}
\]
is known as the \emph{character} of $V$. If we view our representation as a morphism $\rho\colon {\rm GL}_n \to {\rm GL}(V)$, the character $\chi_V$ is equal to $\Tr(\rho(\diag(t_0,\dots,t_{n-1})))$.

So we can refine our main Question \ref{question:main} to:
 \begin{Question} \label{question:repPrecise}
 Let $V=\langle \uMPS(m,n,d) \rangle$, viewed as a ${\rm GL}_n$-representation.
 For every weight $\lambda \in \ZZ^n$, determine the dimension of the weight space $V_{\lambda}$.
 \end{Question}

\subsection{Words, necklaces and bracelets}
  As a warm-up, let us consider the classical representation $V=(\CC^n)^{\ot d}$. Its character is equal to $(t_1+\cdots+t_n)^d$, and by expanding we find that $\dim{V_\lambda}$ is equal to the multinomial coefficient $\binom{d}{\lambda_1, \dots,  \lambda_n}$ if $\sum_{\lambda_i} = d$, and zero otherwise.
We can also see this in terms of coordinates, and this will be useful later:
\begin{Definition}
    A \emph{word} of length $d$ on the alphabet $[n]$ is just an ordered tuple $I=(i_1,\dots,i_d)$, with $i_j \in [n]$. The \emph{weight} of a word $I$ is a tuple $w(I) = (w_0,\dots,w_{n-1}) \in \ZZ^n$, where $w_i$ is the number of entries in $I$ that are equal to $i$.
\end{Definition}
For every word $I=(i_1,\dots, i_d)$, we can define a vector $e_I := e_{i_1} \otimes \cdots \otimes e_{i_d}$. The space $(\CC^n)^{\ot d}$ has a basis given by the $e_I$, where $I$ runs over all words of length $d$ on the alphabet $[n]$. In addition, every $e_I$ is a weight vector of weight $w(I)$. So the dimension of the weight space $V_\lambda$ is the number of words of weight $\lambda$, which is indeed the multinomial coefficient $\binom{d}{\lambda_1, \dots,  \lambda_n}$.

 We move on to the subrepresentations $\Cyc^d(\CC^n)$ and $\Dih^d(\CC^n)$, the natural ambient spaces for uniform matrix product states.

 \begin{Definition}
     A \emph{necklace} (of length $d$ on the alphabet $[n]$) is an equivalence class of words, where two words are equivalent if they agree up to the action $C_d$. A \emph{bracelet} (of length $d$ on the alphabet $[n]$) is an equivalence class of words, where two words are equivalent if they agree up to the action $D_{2d}$. For a fixed necklace or bracelet, all words in the equivalence class clearly have the same weight; this is the \emph{weight} of the necklace or bracelet. We denote by $N(n,d)$, resp. $B(n,d)$, the set of necklaces, resp. bracelets, of length $d$ on $[n]$ and by $N_{\lambda}(n,d)\subset N(n,d)$, resp. $B_{\lambda}(n,d)\subset B(n,d)$, the subset of elements of weight $\lambda\in \ZZ^n$.
 \end{Definition}

 To every necklace $N\in N(n,d)$, we associate a basis vector $e_N := \frac{1}{d}\sum_{\sigma \in C_d}{\sigma \cdot e_I}$, where $I$ is any representative of $N$. Then $\Cyc^d(\CC^n)$ has a basis given by $\{e_N\colon N\in N(n,d)\}$.
 Moreover, $e_N$ is a weight vector of weight $w(N)$, hence we find that the dimension of the weight space of weight $\lambda$ is given by the number $|N_{\lambda}(n,d)|$ of necklaces of weight $\lambda$.

 \begin{Remark}
 The number $|N(n,d)|$ of necklaces of length $d$ on $[n]$ can be counted by the following formula, see for instance \cite[Theorem~7.10]{Stanley}:
 \begin{equation}\label{dimCyc}
        \dim \Cyc^d({\mathbb C}^n)= |N(n,d)|= \frac{1}{d}\sum_{l|d}\phi(l)n^{\frac{d}{l}},
 \end{equation}
 where $\varphi$ is Euler's totient function.
 There is also a formula for $|N_{\lambda}(n,d)| =\dim \Cyc^d({\mathbb C}^n)_{\lambda}$: it is equal to the coefficient of    $x_0^{\lambda_0} \cdots x_{n-1}^{\lambda_{n-1}}$ in the polynomial
\begin{equation}\label{dimCyc2}
\frac{1}{d}\sum_{\ell \mid d}{\big(x_0^{d/\ell}+\cdots+x_{n-1}^{d/\ell}\big)^{\ell}}\varphi\left(\frac{d}{\ell}\right).
\end{equation}
 \end{Remark}
  To every bracelet $b\in B(n,d)$, we associate a basis vector $e_b := \frac{1}{2d}\sum_{\sigma \in D_{2d}}{\sigma \cdot e_I}$, where $I$ is a~representative of $b$. Then $\Dih^d(\CC^n)$ has a basis given by $\{e_b\colon b\in B(n,d)\}$,
  and the dimension of the weight space of weight $\lambda$ is given by the number $|B_{\lambda}(n,d)|$ of bracelets of weight $\lambda$.
  \begin{Remark}
  The number of bracelets of length $d$ on $[n]$ is given by
  \begin{equation}\label{dimDih}
        \dim \Dih^d({\mathbb C}^n)= |B(n,d)|= \begin{cases}
               \displaystyle \frac{1}{2}|N(n,d)| + \frac{1}{4}(n+1)n^{d/2} & \text{for } d \text{ even,}\vspace{1mm}\\
              \displaystyle  \frac{1}{2}|N(n,d)| + \frac{1}{4}n^{(d+1)/2} & \text{for } d \text{ odd.}
        \end{cases}
 \end{equation}
  We only state the formula for $|B_{\lambda}(n,d)|$ in the case of binary bracelets (i.e.\ $n=2$), as that is the only case that is relevant to us:
 \begin{equation}\label{dimDih2}
|B_{(w,d-w)}(2,d)|=\begin{cases}
\displaystyle \left(\frac{1}{2d}\sum_{l \vert \gcd(d,w)}{\varphi(l)\binom{\frac{d}{l}}{\frac{w}{l}}}\right) + \frac{1}{2}\binom{\frac{d}{2}-1}{\frac{w-1}{2}} & \text{for $w$ odd,} \\
\displaystyle \left(\frac{1}{2d}\sum_{l \vert \gcd(d,w)}{\varphi(l)\binom{\frac{d}{l}}{\frac{w}{l}}}\right) + \frac{1}{2}\binom{\frac{d}{2}}{\frac{w}{2}} & \text{for $w$ even.}
\end{cases}
\end{equation}
\end{Remark}
Equations \eqref{dimCyc}--\eqref{dimDih2} can be derived from the P\'olya enumeration theorem \cite{polya2012combinatorial} applied to the action of the
 cyclic and the dihedral group, respectively.

  \section{Computations} \label{sec:comp}

 In this section, we describe how to computationally answer Question~\ref{question:repPrecise} for fixed parameters. We focus on the smallest interesting case $m=n=2$.
 In this case we are dealing with representations of ${\rm GL}_2$, so the weights are in~$\ZZ^2$. Moreover, the only occurring weights in $\big(\CC^2\big)^{\ot d}$ are $t_0^{w}t_1^{d-w}$ for $w=0,\dots,d$. For subrepresentations of $\big(\CC^2\big)^{\ot d}$, we will abbreviate the weight spaces $V_{(w,d-w)}$ to~$V_w$. Our goal is to determine the dimension of the weight spaces $\langle \uMPS(2,2,d) \rangle_w$.

 All of our dimension counts in this section and the next use the following easy observation.
 \begin{observation} \label{obs:dimLinSpan}
       For $p_1, \dots, p_N \in \CC[y_1,\dots,y_s]$ polynomials, and $X$ the image of the polynomial map
 \begin{gather} \label{eq:imageOfAPolynomialMap}
 \begin{split}
     & \CC^s\to \CC^N ,\\
    &  (y_1, \dots, y_s) \mapsto (p_1(y_1,\dots, y_s), \dots, p_N(y_1,\dots, y_s)),
      \end{split}
 \end{gather}
 a linear equation $\sum{\alpha_ix_i}$ vanishes on $X$ if and only if the identity $\sum{\alpha_ip_i} = 0$ holds in the polynomial ring $\CC[y_1,\dots,y_s]$. In particular, the dimension of the linear span of $X$ is equal to the dimension of the subspace of $\CC[y_1,\dots,y_s]$ spanned by the $p_i$'s.
 \end{observation}

 \subsection{The trace parametrization} If we directly use Definition~\ref{def: umps}, we see that $\uMPS(2,2,d)$ is the closed image of a polynomial map~$\CC^8 \to \Dih^d\big(\CC^2\big)$. However, in this specific case there is an alternative parametrization by~$\CC^5$ instead of~$\CC^8$: the \emph{trace parametrization}, which appears to be computationally more efficient in practice. It is based on the connection between uniform matrix product states and invariant theory of matrices.

 \begin{Definition} \label{def:TraceAlgebra}
    Let
 \[
 R=\CC\big[\big(a_{ij}^k\big)_{1\leq i,j \leq m; 1 \leq k \leq n}\big]
  \]be the polynomial ring in $m^2n$ variables, and for $k=0, \dots, n-1$, let $A_k := \big(a_{ij}^k\big)_{1\leq i,j \leq m}$ be generic $m \times m$ matrices.
    The \emph{trace algebra} $\mathcal{C}_{m,n}$ is the subalgebra of $R$ generated by the polynomials $\Tr(A_{i_1}\cdots A_{i_s})$, where $(i_1,\dots,i_s)$ runs over all words (or equivalently: necklaces) in~$[n]$.
 \end{Definition}
 \begin{Remark}
    The trace algebra is precisely the subring of $R$ consisting of all elements that are invariant under simultaneous conjugation:
\[
    f \in \mathcal{C}_{m,n} \iff f\big(P^{-1}A_0P,\dots,P^{-1}A_{n-1}P\big) = f(A_0,\dots,A_{n-1}) \quad  \forall P \in {\rm GL}_m.
\]
    This is known as the \emph{first fundamental theorem} in the invariant theory of matrices \cite{procesi1976invariant, sibirskii1968algebraic}.
 \end{Remark}
 \begin{Proposition}[{\cite[Corollary 2]{sibirskii1968algebraic}}]
The trace algebra $\mathcal{C}_{2,2}$ is generated by the following five polynomials:
\[
\Tr(A_0), \quad \Tr(A_1),\quad \Tr\big(A_0^2\big),\quad \Tr(A_0A_1), \quad \Tr\big(A_1^2\big),
\]
and moreover, there are no polynomial relations between these generators.
 \end{Proposition}
 \begin{Corollary}
     For every bracelet $b=(b_1,\dots,b_k)$, there is a unique  polynomial
     \[
     P_b(T_0,T_1,T_{00},T_{01},T_{11}) \in \CC[T_0,T_1,T_{00},T_{01},T_{11}]
     \]
     such that for every pair $(A_0,A_1)$ of $2\times2$ matrices, the following equality holds:
     \begin{equation*} 
         \Tr(A_{b_1}\cdots A_{b_k}) = P_b\big(\Tr(A_0),\Tr(A_1), \Tr\big(A_0^2\big), \Tr(A_0A_1), \Tr\big(A_1^2\big)\big).
     \end{equation*}
 \end{Corollary}
 \begin{Remark} \label{rmk:multiHomog}
    If we give the ring $\CC[T_0,T_1,T_{00},T_{01},T_{11}]$ a grading by putting $\deg(T_0)=\deg(T_1)\allowbreak =1$ and $\deg(T_{00})=\deg(T_{01})=\deg(T_{11})=2$, then the polynomial $P_b$ is homogeneous of degree $\operatorname{length}(b)$.
 \end{Remark}

 The above means that $\uMPS(2,2,d)$ is the image of the polynomial map
 \begin{align} \label{eq:traceParMap}
 \begin{split}
     \psi\colon \  \CC^5 &\to \Dih^d\big(\CC^2\big), \\
     (T_0,T_1,T_{00},T_{01},T_{11})& \mapsto \sum_{b}{P_b(T_0,T_1,T_{00},T_{01},T_{11})e_b},
     \end{split}
 \end{align}
 where $b$ runs over all bracelets of length $d$. This is the trace parametrization.

 In order to compute the polynomials $P_b$, for bracelets of length $3$, one verifies that
 \begin{alignat*}{3}
 &   P_{000}= -\frac{1}{2}T_0^3 + \frac{3}{2}T_0T_{00},\qquad& &
    P_{100}= -\frac{1}{2}T_0^2T_1 + \frac{1}{2}T_1T_{00}+ T_0T_{01}, &\\
&    P_{110}= -\frac{1}{2}T_0T_1^2 + \frac{1}{2}T_0T_{11}+ T_1T_{01},\qquad & &
    P_{111}= -\frac{1}{2}T_1^3 + \frac{3}{2}T_1T_{11}.&
 \end{alignat*}

For bracelets of length $\geq 4$, we can inductively use the following identity \cite{sibirskii1968algebraic}, which holds for every quadruple $(A,B,C,D)$ of $2 \times 2$ matrices:
\begin{gather*}
    2\Tr(ABCD) = \Tr(A) (\Tr(BCD) - \Tr(B)\Tr(CD) ) + \Tr(B) (\Tr(CDA) - \Tr(C)\Tr(DA) ) \nonumber\\
  \hphantom{2\Tr(ABCD) =}{}   + \Tr(C) (\Tr(DAB) - \Tr(D)\Tr(AB) ) \nonumber \\
  \hphantom{2\Tr(ABCD) =}{} + \Tr(D) (\Tr(ABC) - \Tr(A)\Tr(BC) ) -\Tr(AC)\Tr(BD)\\ 
\hphantom{2\Tr(ABCD) =}{} + \Tr(AB)\Tr(CD) + \Tr(AD)\Tr(BC)  + \Tr(A)\Tr(B)\Tr(C)\Tr(D).  \nonumber
\end{gather*}

\subsection{Computing the character}\label{subsec: char}
 The weight space
 $\langle \uMPS(2,2,d) \rangle_w$
 is the linear span of the image of the map
\begin{align*}
     \CC^5 &\to \Dih^d\big(\CC^2\big)_w ,\\
     (T_0,T_1,T_{00},T_{01},T_{11})& \mapsto \sum_{b}{P_b(T_0,T_1,T_{00},T_{01},T_{11})e_b},
 \end{align*}
 where $b$ ranges over all bracelets of weight $w$. By Observation~\ref{obs:dimLinSpan}, we need to compute the dimension of the linear subspace of $\CC[T_0,T_1,T_{00},T_{01},T_{11}]$ spanned by the $P_b$'s. This can be computed by putting the coefficients of the $P_b$'s in a matrix and computing its rank.

Putting everything together, we obtain Algorithm~\ref{alg:Linear span }, which for a given $d$ computes the character (in particular the dimension) of $\langle \uMPS(2,2,d) \rangle$.
We implemented this algorithm in SageMath~\cite{code}.

\RestyleAlgo{ruled}

\begin{algorithm}[th!]
\caption{Linear span of $\langle \uMPS(2,2,d) \rangle$}\label{alg:Linear span }
\KwData{$d$}
\KwResult{Character of the representation $\langle \uMPS(2,2,d) \rangle$}
$T_0 \gets \Tr(A_0)$; $T_1 \gets \Tr(A_1)$; $T_{00} \gets \Tr\big(A_0^2\big)$; $T_{01} \gets \Tr(A_0A_1)$; $T_{11} \gets \Tr\big(A_1^2\big)$\;
\For{$\ell=3$ \KwTo d}{
$bracelets$ = bracelets of length $\ell$\;
\For{$b$ in $bracelets$}{
P[b] $\gets$ $\Tr(A_{b_1}\cdots A_{b_\ell})$, expressed in the $T_i$\;
}
}
\For{$w=0$ \KwTo $d$}{
List the monomials $\bfy^{\alpha_1},\dots,\bfy^{\alpha_t}$ appearing in the $P[b]$, where $b$ ranges over all bracelets of weight $w$\;
Write $P[b]=\sum_j \beta_{b,j} \bfy^{\alpha_j}$\;
Compute the rank of the matrix $(\beta_{b,j})_{b,j}$\;
}
\end{algorithm}

The results are summarized in Table~\ref{table:character}, where we write
$D_w := \dim \langle \uMPS(2,2,d) \rangle_{w}$.
For $d < 8$, the space $\langle \uMPS(2,2,d) \rangle$ is equal to the ambient space $\Dih^d\big(\CC^2\big)$.

 \begin{table}[th]\centering
	\caption{Character of $\langle \uMPS(2,2,d) \rangle$. Since $D_{w}=D_{d-w}$ (more generally: for every ${\rm GL}_2$-re\-pre\-sen\-tation~$V$ we have that $V_{(b_0,b_1)}=V_{(b_1,b_0)}$), we only list $D_w$ for $w \leq \big\lceil \frac{d}{2}\big\rceil$. The dimension of $\langle \uMPS(2,2,d) \rangle$ equals twice the sum of the $D_w$'s if $d$ is odd, and twice the sum of the $D_w$'s minus the last $D_w$ if $d$ is even.}\label{table:character}

\vspace{-5mm}

$$
	\begin{array}{@{\,}l|l@{\,\,}l@{\,\,}l@{\,\,}l@{\,\,}l@{\,\,}l@{\,\,}l@{\,\,}l@{\,\,}l@{\,\,}l@{\,\,}l|l|l@{\,}}
	\toprule
	d & D_0	& D_1 	& D_2 	& D_3 	& D_4 & D_5 & D_6 & D_7 & D_8 & D_9 & D_{10} &\dim \langle\uMPS(2,2,d)\rangle & \dim \Dih^d\big(\CC^2\big)\\
	\midrule
	8 & 1 	& 1 	& 4 	& 5 	& 7  	&  &   &  &  &&& 29 & 30\\
	9 & 1	& 1 	& 4 	& 6 	& 8  	&  &  & & &&& 40 & 46\\
	10 & 1 	& 1 	& 5 	& 7 	& 11 	& 11 &  & &   &  && 61 & 78\\
	11 &  1 & 1 	& 5 	& 8 	& 12 	& 14 &  & &   &  && 82 & 126\\
	12 &  1	& 1 	& 6 	& 9 	& 15 	& 17 & 20 &  &  &  && 118 & 224\\
	13 &  1	& 1 	& 6 	& 10 	& 16 	& 20 & 23 &  &  &  && 154 & 380\\
	14 & 1 	& 1 	& 7		& 11	& 19 	& 23 & 29 & 29&  &  && 211 & 687\\
	15 & 1 	& 1 	& 7		& 12 	& 20 	& 26 & 32 & 35&  &  && 268 & 1224\\
	16 & 1 	& 1 	& 8		& 13 	& 23 	& 29 & 38 &41 & 45 &  && 353 & 2250\\
	17 & 1  & 1     & 8     & 14    & 24    & 32 & 41 & 47 & 51  &  && 438 & 4112\\
	18 & 1  & 1     & 9     & 15    & 27    & 35 & 47 & 53 & 61 & 61 && 559 & 7685 \\
	19 & 1  & 1     & 9     & 16    & 28    & 38 & 50 & 59 & 67 & 71 && 680 & 14310 \\
    20 & 1  & 1     & 10    & 17    & 31    & 41 & 56 & 65 & 77 & 81 & 86 & 846 & 27012\\
	\bottomrule
	\end{array}$$
\vspace{-2mm}
\end{table}

 Based on these computations, we make the following conjecture:

\begin{Conjecture}\label{conj: dim}
	\begin{gather*}
	\dim\langle \uMPS(2,2,d) \rangle_{w} \\
\qquad{} = \begin{cases}
	\displaystyle    1 + \frac{d(v-1)v}{2} - \frac{2(v-1)v(2v-1)}{3} + v\left\lfloor \frac{d}{2}\right\rfloor - 2v^2 + v & \text{ for $w = 2v$,} \vspace{1mm}\\
	 \displaystyle   1 + \frac{dv(v+1)}{2} - \frac{2v(v+1)(2v+1)}{3}             & \text{for $w = 2v+1$.}
	\end{cases}
	\end{gather*}
\end{Conjecture}

This would in particular imply:
\begin{Conjecture}[Corollary of Conjecture \ref{conj: dim}] \label{conj:dimcor}
\[
	\dim \langle \uMPS(2,2,d)\rangle = \begin{cases}
	\displaystyle\frac{1}{192}\big(d^4-4d^2+192d+192\big) & \text{ for $d$ even,}\vspace{1mm}\\
	\displaystyle\frac{1}{192}\big(d^4-10d^2+192d+201\big) & \text{ for $d$ odd.}
	\end{cases}
\]
\end{Conjecture}

\subsection{Higher degree equations}

Equations of tensor varieties such as secant varieties of the Segre variety or Veronese variety are well known to be hard to compute. For uniform matrix product states, Critch and Morton gave a~complete description of the ideal of $\uMPS(2,2,4)$ and $\uMPS(2,2,5)$ and several linear equation of $\uMPS(2,2,d)$ are given for $d$ until $12$, cf.~\cite{Critch2014AlgebraicGO}. The generators of the ideal of $\uMPS(2,2,d)$ for $d=4,5,6$ are given in~\cite{Czaplinski2019UniformMP}.

Our algorithm from the previous section computes the linear span of $\uMPS(2,2,d)$, which is equivalent to computing the degree $1$ part of its defining ideal. With some small modifications, one can obtain an algorithm to compute the degree $k$ part of the defining ideal, viewed as a~${\rm GL}_2$-representation. In this paper we will only give a brief sketch of this algorithm. For a~more detailed account, including the use of raising operators to speed up the algorithm, we refer the reader to \cite{ClaudiasThesis}.

We first consider a slightly more general setting:
let $X$ be any variety that is given as the closed image of a homogeneous polynomial map $p\colon \CC^s \to \CC^N$ as in (\ref{eq:imageOfAPolynomialMap}), and let $I$ be its (homogeneous) defining ideal. Suppose that $p$ is compatible with a given ${\rm GL}_n$-action on~$\CC^s$ and~$\CC^N$. Then for every $k \in {\mathbb N}$, the degree $k$ part $I_k$ is naturally a ${\rm GL}_n$-representation. Fix a~weight $w$ and consider the map
\[
p^k_w\colon \  \CC^s \xrightarrow{p} \CC^N \xrightarrow{\nu_k} S^k\CC^N \xrightarrow{\pi_w} \big(S^k\CC^N\big)_w,
\]
where $\nu_k$ is the $k$'th Veronese embedding and $\pi_w$ is the projection onto the weight space. The latter map commutes with the action of the torus.
Then, the weight space $I_{k,w} \subseteq \big(S^k\CC^N\big)_w^*$ is equal to
\[ \big\langle \Im(p^k_w) \big\rangle^\perp \cong \big({S^k \CC^N} / {\big\langle \Im\big( p^k_w\big) \big\rangle} \big)^*,\]
see for example  \cite[Proposition~4.4.1.1]{landsberg2012tensors}.
The character of this representation is given by the difference of the characters of $\big\langle \Im\big(p^k_w\big) \big\rangle$ and $S^k \CC^N$. The latter character can be computed from the character of $\CC^N$, and hence we are left to compute the character of $\big\langle \Im\big(p^k_w\big) \big\rangle$.

In the case $X=\uMPS(2,2,d)$, we can compute the character of $\big\langle \Im\big(p^k_w\big) \big\rangle$ using a minor modification of Algorithm~\ref{alg:Linear span }, where in the second loop we replace the collection of polynomials
$\{P[b] \mid \operatorname{weight}(b) = w \}$
with their $k$-fold products
\[ \bigg\{\prod_{j=1}^{k} P[b_j] \mid \sum_j \operatorname{weight}(b_j) = w \bigg\}.\]

Using this algorithm, we computed $I(\uMPS(2,2,d))_2$ for $d\leq 10$ and $I(\uMPS(2,2,d))_3$ for $d\leq 9$. Our code is available at \cite{code}, and our results are summarized in Tables~\ref{table:character2} and~\ref{table:character3}.

 \begin{table}[th!]\centering
	\caption{Character of $I(\uMPS(2,2,d))^*_2$. Since $D_{w}=D_{2d-w}$ (again: for every ${\rm GL}_2$-representation $V$ we have that $V_{(b_0,b_1)}=V_{(b_1,b_0)}$),  we only list $D_w$ for $w \leq d$.}  \label{table:character2}

\vspace{-5mm}

$$
	\begin{array}{l|llllllll}
	\toprule
	d & D_3 	& D_4 & D_5 & D_6 & D_7 & D_8 & D_9 & D_{10}\\
	\midrule
	6 & 0   & 1  & 1  & 2   &     &     &     &     \\
	7 & 0   & 1  & 3  & 6   & 7   &     &     &     \\
	8 & 0   & 5  & 10 & 25  & 32  & 42  &     &     \\
	9 &  1  & 7  & 21 &  48 & 79  & 110 & 119 &     \\
	10 &  1 & 14 & 38 & 100 & 176 & 290 & 360 & 408 \\
	\bottomrule
	\end{array}$$
\vspace{-2mm}
\end{table}

 \begin{table}[th!]\centering
	\caption{Character of $I(\uMPS(2,2,d))^*_3$. Since $D_{w}=D_{3d-w}$, we only list $D_w$ for $w \leq \big\lceil \frac{3d}{2}\big\rceil$.}  \label{table:character3}

\vspace{-5mm}

$$
	\begin{array}{l|lllllllllll}
	\toprule
	d & D_3 & D_4 & D_5 & D_6 & D_7 & D_8 & D_9 & D_{10} & D_{11} & D_{12} & D_{13}\\
	\midrule
	6 & 0 & 1 &  2 &  8 &  11 &  17 &  17 &      &      &      &      \\
	7 & 0 & 1 &  4 & 15 &  29 &  49 &  67 &   77 &      &      &      \\
	8 & 0 & 5 & 14 & 51 & 101 & 198 & 292 &  414 & 478  & 532  &      \\
	9 & 1 & 7 & 26 & 83 & 191 & 388 & 671 & 1039 & 1431 & 1784 & 1983 \\
	\bottomrule
	\end{array}$$
\vspace{-2mm}
\end{table}

\section{Results} \label{sec:results}

 \subsection{Linear relations via Cayley--Hamilton} \label{sec:CH}
We recall the classical Cayley--Hamilton theorem.
We use this result in order to find linear equations for ${\uMPS(m,n,d+k)}$, $k\geq m$ based on linear equations for ${\uMPS(m,n,N)}$, $N=d,d+1,\dots,d+m-1$.

\begin{Theorem}[Cayley--Hamilton]\label{theorem: CH}
Let $A\in {\mathbb C}^{m\times m}$ be a $m\times m$ complex matrix and $p_A(\lambda)=\det(\lambda \,{\rm Id}_m -A)$ its characteristic polynomial. Then $p_A(A)=0$ where $p_A(\cdot)$ is the corresponding matrix-valued polynomial.
\end{Theorem}
In fact, the only thing we will use is the following statement, which (since $\deg p_A = m$) immediately follows from Theorem~\ref{theorem: CH}.
\begin{Corollary}\label{cor: CH}
	Let $A\in {\mathbb C}^{m\times m}$. Then $A^k$, for $k\geq m$ can be written as a linear combination of its previous powers.
\end{Corollary}
\begin{Lemma}\label{lemma: CH}
    Let $c = (c_1,\dots,c_s) \in \CC^s$ be a vector of coefficients and $\big\{i_\ell^{j}\big\}_{1\leq\ell\leq d, 1 \leq j \leq s}$ be indices; with $i_\ell^{j} \in [n]$.
    Assume that for every $n$-tuple $(A_0,\dots,A_{n-1})$ of $m \times m$ matrices and every $k < m$ the following identity holds:
    \begin{equation} \label{eq:lemCH}
        \sum_{j=1}^s c_j \Tr\big(A_{i_1^j} \cdots A_{i_d^j}A_0^k\big)=0.
    \end{equation}
	Then the same identity holds for arbitrary $k \in {\mathbb N}$.
\end{Lemma}
\begin{proof}
	We use induction on $k \geq m$. By Corollary \ref{cor: CH}, $A_0^k$ can be written as a linear combination of its previous powers $A_0^l$, for $l=0,\dots,m-1$. Therefore, we have
	\begin{align*}
	    \sum_{j=1}^{s}c_j \Tr\big(A_{i_1^j}\cdots A_{i_d^j} A_{0}^k\big)
	    & = \sum_{j=1}^{s}c_j \Tr\bigg[A_{i_1^j}\cdots A_{i_d^j} \bigg(\sum_{l=0}^{m-1}\gamma_l A_0^l\bigg)\bigg]\\
	    & = \sum_{j=1}^s c_j\sum_{l=0}^{m-1} \gamma_l \big(\Tr\big(A_{i_1^j}\cdots A_{i_d^j} A_{0}^l\big)\big)\\
	    & = \sum_{l=0}^{m-1} \gamma_l \bigg(\sum_{j=1}^s c_j \big(\Tr\big(A_{i_1^j}\cdots A_{i_d^j} A_{0}^l\big)\big)\bigg)=0.\tag*{\qed}
	\end{align*}\renewcommand{\qed}{}
\end{proof}

The usefulness of Lemma~\ref{lemma: CH} stems from the fact that one can find expressions of the form~(\ref{eq:lemCH}) which are trivial for small $k$, in the sense that they follow from the cyclic invariance of the trace, but nontrivial for large $k$. We illustrate this in the example below.
\begin{Example}
We show that for any $2 \times 2$ matrices $A_0$, $A_1$, $A_2$, $A_3$ and any $k \geq 0$, the following identity holds
    \begin{gather*}
		 \Tr\big(A_1A_2A_0A_3A_0^k\big)+\Tr\big(A_2A_3A_0A_1A_0^k\big)+\Tr\big(A_3A_1A_0A_2A_0^k\big)\\
		\qquad{}= \Tr\big(A_1A_0A_2A_3A_0^k\big)+\Tr\big(A_2A_0A_3A_1A_0^k\big)+\Tr\big(A_3A_0A_1A_2A_0^k\big).
		\end{gather*}
	By the above argument, it suffices to show the identity for $k=0$ and $k=1$. But these both follow from cyclic invariance of the trace:
	\begin{gather*}
	 \Cline[blue]{\Tr(A_1A_2A_0A_3)}+\Cline[red]{\Tr(A_2A_3A_0A_1)}+\Cline[cyan]{\Tr(A_3A_1A_0A_2)}\\
\qquad{}	= \Cline[cyan]{\Tr(A_1A_0A_2A_3)}+\Cline[blue]{\Tr(A_2A_0A_3A_1)}+\Cline[red]{\Tr(A_3A_0A_1A_2)},\\
	\Cline[blue]{\Tr(A_1A_2A_0A_3A_0)}+\Cline[red]{\Tr(A_2A_3A_0A_1A_0)}+\Cline[cyan]{\Tr(A_3A_1A_0A_2A_0)}\\
	\qquad{} =\Cline[red]{\Tr(A_1A_0A_2A_3A_0)}+\Cline[cyan]{\Tr(A_2A_0A_3A_1A_0)}+\Cline[blue]{\Tr(A_3A_0A_1A_2A_0)}.
	\end{gather*}
	This immediately gives us a nontrivial linear relation on $\uMPS(2,4,d)$, for $d \geq 6$.
	But we can also find linear relations on $\uMPS(2,2,d)$. For instance, if we put $k=2$, $A_2=A_1^2$ and $A_3=A_0A_1$, we find
	\begin{gather*}
		 \Tr\big(A_1^3A_0^2A_1A_0^2\big)+\Tr\big(A_1^2A_0A_1A_0A_1A_0^2\big)+\Tr\big(A_1^2A_0^3A_1^2A_0\big)\\
		\qquad {}= \Tr\big(A_1^2A_0A_1A_0^2A_1A_0\big)+\Tr\big(A_1^2A_0^2A_1^2A_0^2\big)+\Tr\big(A_1^3A_0^3A_1A_0\big),
		\end{gather*}
	which is the unique linear relation on $\uMPS(2,2,8)$ that doesn't follow from dihedral symmetry.
\end{Example}

\begin{Theorem} \label{thm:linRel}
	Let $A_0,\dots,A_m,B$ be $m \times m$ matrices. Then for every $\ell \in {\mathbb N}$ it holds that
	\begin{gather} \label{eq:linRel}
	\sum_{\sigma \in \mathfrak{S}_m, \tau \in C_{m+1}}\!\!\!\!{\sgn(\sigma)\sgn(\tau)\Tr\big(A_{\tau(0)}B^{\sigma(0)}A_{\tau(1)}B^{\sigma(1)} \cdots A_{\tau(m-1)}B^{\sigma(m-1)}A_{\tau(m)}B^{\ell}\big)}=0.\!\!\!
	\end{gather}
	Here $\mathfrak{S}_m$ is the symmetric group acting on $\{0,1,\dots,m-1\}$, and $C_{m+1}$ is the cyclic group acting on $\{0,1,\dots,m\}$.
\end{Theorem}
\begin{proof}
	We will first show the statement for $\ell \in \{0,1,\dots,m-1\}$. So let us fix such an $\ell$. We will write
	\[T(\sigma,\tau) := \Tr\big(A_{\tau(0)}B^{\sigma(0)}A_{\tau(1)}B^{\sigma(1)} \cdots A_{\tau(m-1)}B^{\sigma(m-1)}A_{\tau(m)}B^{\ell}\big).\]
	Let us write $c_{a}$ for the permutation that cyclically permutes the first $a$ elements. Precisely
	\[c_{a}(i) =
	\begin{cases}
		i+1 & \text{for } i<a-1, \\
		0 & \text{for } i=a-1, \\
		i & \text{for } i>a-1.
	\end{cases}\]

{\bf Step 1.}
	For $\sigma \in \mathfrak{S}_m$ and $\tau \in C_{m+1}$, we define{\samepage
	\begin{gather*}
	\tilde{\sigma} := \sigma \circ c_{\sigma^{-1}(\ell)+1}^{-1} \circ c_{m}^{\sigma^{-1}(\ell)+1}, \qquad
	\tilde{\tau} := \tau \circ c_{m+1}^{\sigma^{-1}(\ell)+1}.
	\end{gather*}
	We have $T(\sigma,\tau) = T(\tilde{\sigma}, \tilde{\tau})$.}
	
	Indeed, if we write $k=\sigma^{-1}(\ell)$, then
	\begin{align*}
	T(\sigma,\tau) &= \Tr\big(A_{\tau(0)}B^{\sigma(0)}\cdots A_{\tau(k)}B^{\sigma(k)}A_{\tau(k+1)}B^{\sigma(k+1)} \cdots A_{\tau(m-1)}B^{\sigma(m-1)}A_{\tau(m)}B^{\sigma(k)}\big)\\
	&=\Tr\big(A_{\tau(k+1)}B^{\sigma(k+1)} \cdots A_{\tau(m)}B^{\sigma(k)}A_{\tau(0)}B^{\sigma(0)} \cdots A_{\tau(k-1)}B^{\sigma(k-1)}A_{\tau(k)}B^{\sigma(k)}\big)\\
	&=T(\tilde{\sigma}, \tilde{\tau}).
	\end{align*}
	Where for the last step, note that
	\begin{itemize}\itemsep=0pt
		\item $k+1 = c_{m+1}^{k+1}(0)$, \dots, $m=c_{m+1}^{k+1}(m-k-1)$, $0=c_{m+1}^{k+1}(m-k)$, \dots, $k=c_{m+1}^{k+1}(m)$,
		\item $k+1=c_{k+1}^{-1}\big(c_m^{k+1}(0)\big)$, \dots, $m-1=c_{k+1}^{-1}\big(c_m^{k+1}(m-k-2)\big)$, $k=c_{k+1}^{-1}\big(c_m^{k+1}(m-k-1)\big)$, $0=c_{k+1}^{-1}\big(c_m^{k+1}(m-k)\big)$, \dots, $k-1=c_{k+1}^{-1}\big(c_m^{k+1}(m-1)\big)$.
	\end{itemize}

    {\bf Step 2.}
	Note that the assignment
	\begin{align*}
	\mathfrak{S}_m \times C_{m+1} & \to \mathfrak{S}_m \times C_{m+1}, \\
	(\sigma,\tau) & \mapsto (\tilde{\sigma}, \tilde{\tau})
	\end{align*}
	is an involution. Indeed,
	we have
	\begin{gather*}
		\tilde{\sigma}^{-1}(\ell) = c_m^{-\sigma^{-1}(\ell)-1}\big(c_{\sigma^{-1}(\ell)+1}\big(\sigma^{-1}(\ell)\big)\big)
		= c_m^{-\sigma^{-1}(\ell)-1}(0)
		= m-\sigma^{-1}(\ell)-1.
	\end{gather*}
	So, again writing $k=\sigma^{-1}(\ell)$,
	\begin{gather*}
	\tilde{\tilde{\sigma}} = \tilde{\sigma} \circ c_{\tilde{\sigma}^{-1}(\ell)+1}^{-1} \circ c_{m}^{\tilde{\sigma}^{-1}(\ell)+1}
	=\sigma \circ c_{k+1}^{-1} \circ c_{m}^{k+1} \circ c_{m-k}^{-1} \circ c_{m}^{m-k}
	= \sigma.
	\end{gather*}
	To see the last equality:
	\begin{itemize}\itemsep=0pt
		\item for $i<k$: $c_{k+1}^{-1}\big(c_{m}^{k+1}\big(c_{m-k}^{-1}\big(c_{m}^{m-k}(i)\big)\big)\big) = c_{k+1}^{-1}\big(c_{m}^{k+1}\big(c_{m-k}^{-1}(m-k+i)\big)\big) =  c_{k+1}^{-1}\big(c_{m}^{k+1}(m-k+i)\big)=c_{k+1}^{-1}(i+1)=i$,
		\item for $i=k$: $c_{k+1}^{-1}\big(c_{m}^{k+1}\big(c_{m-k}^{-1}(c_{m}^{m-k}(k))\big)\big) = c_{k+1}^{-1}\big(c_{m}^{k+1}\big(c_{m-k}^{-1}(0)\big)\big) = c_{k+1}^{-1}\big(c_{m}^{k+1}(m-k-1)\big)=c_{k+1}^{-1}(0)=k$,
		\item for $i>k$: $c_{k+1}^{-1}\big(c_{m}^{k+1}\big(c_{m-k}^{-1}(c_{m}^{m-k}(i))\big)\big) = c_{k+1}^{-1}\big(c_{m}^{k+1}\big(c_{m-k}^{-1}(i-k)\big)\big) =  c_{k+1}^{-1}\big(c_{m}^{k+1}(i-k-1)\big)=c_{k+1}^{-1}(i)=i$.
	\end{itemize}
	And furthermore $\tilde{\tilde{\tau}} = \tau \circ c_{m+1}^{k+1} \circ c_{m+1}^{m-k} = \tau$.
	
	{\bf Step 3.} Note that
	\begin{gather*}
	\sgn(\tilde{\sigma})\sgn(\tilde{\tau})  = (-1)^{k+(k+1)(m-1)+(k+1)m}\sgn(\sigma)\sgn(\tau)
	 =-\sgn(\sigma)\sgn(\tau).
	\end{gather*}
	
	From Step 1 we have that $T(\sigma,\tau) = T(\tilde{\sigma},\tilde{\tau})$. There will be cancellations of terms in~(\ref{eq:linRel}) as  $\sgn(\tilde{\sigma})\sgn(\tilde{\tau}) = -\sgn(\sigma)\sgn(\tau)$ from Step~3. Finally using Step~2 we ensure that all terms will cancel out therefore establishing the given identity for $ 0 \leq l \leq m-1$. Now using Lemma~\ref{lemma: CH}, we conclude that the given identity~(\ref{eq:linRel}) holds for all $l \in \mathbb{N}$.
\end{proof}

\begin{Corollary}\label{cor: lin span}
	If $n \geq 3$ and $d \geq \frac{(m+1)(m+2)}{2}$, then ${\uMPS(m,n,d)}$ is contained in a proper linear subspace of the space of cyclically invariant tensors.
\end{Corollary}
\begin{proof}
Let $\ell \geq m$ and let $\SG_{m+1}$ denote the symmetric group acting on $\{0,1,\dots,m-1,\ell\}$. As the trace is invariant under cyclic permutations, for every $\tau \in C_{m+1}$, \[\Tr\big(A_{\tau(0)}B^{\sigma(0)}A_{\tau(1)}B^{\sigma(1)} \cdots A_{\tau(m-1)}B^{\sigma(m-1)}A_{\tau(m)}B^{\ell}\big)\] can be written as \[\Tr\big(A_{0}B^{\sigma'(0)}A_{1}B^{\sigma'(1)} \cdots A_{m-1}B^{\sigma'(m-1)}A_{m}B^{\sigma'(\ell)}\big)\] fixing all positions of $A_i$'s and permuting $B^j$'s according to some $\sigma' \in \mathfrak{S}_{m+1}$. If for a moment we identify $\{0,1,\dots,m-1,\ell\}$ with $[m]$, we have $\sigma'=\sigma \circ \tau^{-1}$. Therefore, we can rewrite (\ref{eq:linRel}) as follows:
\begin{equation} \label{eq:linRelAlt}
	\sum_{\sigma \in \mathfrak{S}_{m+1}}{\sgn(\sigma)\Tr\big(A_{0}B^{\sigma(0)}A_{1}B^{\sigma(1)} \cdots A_{m-1}B^{\sigma(m-1)}A_{m}B^{\sigma(\ell)}\big)}=0.
\end{equation}
Let $X_0$, $X_1$, $X_2$ be $m \times m$ matrices, and in (\ref{eq:linRelAlt}) substitute $A_0=X_0$, $B=X_1$, and $A_i = X_2$ for $i=1,\dots,m$. Note that even after that substitution, the ternary bracelets corresponding to the $(m+1)!$ terms in~(\ref{eq:linRelAlt}) are all distinct. Hence no two terms will cancel, and we get a~nontrivial linear relation on $\uMPS(m,3,d)$, where $d=1+2+\cdots+(m-1)+ \ell + (m+1) \geq 1+2+\cdots+ (m+1) = \binom{m+2}{2}$.
\end{proof}

\begin{Remark}
With a bit more care, one can also get nontrivial relations on $\uMPS(m,2,d)$ in this way. For instance if we take $\ell=m$ and in~(\ref{eq:linRelAlt}) we substitute $A_0=X_0X_1^{m+1}X_0$, $B=X_1$, and $A_i = X_0$ for $i=1,\dots,m$, one verifies that again no terms cancel, and hence we found a~nontrivial linear relation on~$\uMPS(m,2,d)$, where $d=\binom{m+3}{2}$.
\end{Remark}

\subsection[Linear equations for uMPS(2,2,d)]{Linear equations for $\boldsymbol{\uMPS(2,2,d)}$} \label{sec:lineq22d}

From the trace parametrization, we can give an upper bound on $\dim \langle \uMPS(2,2,d) \rangle$.

\begin{Theorem}
For every $d \in {\mathbb N}$, we have the inequality
\[   \dim \langle \uMPS(2,2,d) \rangle \leq
\begin{cases}
      \displaystyle  \frac{1}{192}(d + 6)(d + 4)^2(d + 2) & \text{for $d$ even}, \vspace{1mm}\\
     \displaystyle   \frac{1}{192}(d + 7)(d + 5)(d + 3)(d + 1) & \text{for $d$ odd}.
    \end{cases}
\]
\end{Theorem}
\begin{proof}
    It follows from (\ref{eq:traceParMap}) and Remark~\ref{rmk:multiHomog} that $\dim \langle \uMPS(2,2,d) \rangle$ can be at most the number of degree $d$ monomials in $\CC[T_0,T_1,T_{00},T_{01},T_{11}]$. Counting these monomials gives the above formula.
\end{proof}

Note that asymptotically for $d \to \infty$, the above bound agrees with our conjectured formula in Conjecture \ref{conj:dimcor}.

As in the previous section, we abbreviate ``weight $\lambda=(w,d-w)\in \ZZ^2$'' to ``weight $w$''. In the rest of this section, we provide a proof of Conjecture \ref{conj: dim} in the cases of weight $w=0,1,2,3$.

Consider the parametrization of $\uMPS(2,2,d)$ in coordinates
\begin{align*}
\phi\colon \ \big({\mathbb C}^{2\times 2}\big)^2 &\to \Dih^d\big({\mathbb C}^2\big),\\
(A_0,A_1) & \mapsto \big(\Tr\big(A_0^d\big),\Tr\big(A_0^{d-1}A_1\big),\dots,\Tr\big(A_1^d\big)\big).
\end{align*}
It is in particular a polynomial map in the unknown entries of the matrices $A_0,A_1\in {\mathbb C}^{2\times 2}$, that we denote by
\[
A_0=
\begin{pmatrix}
a_1 & a_2 \\
a_3 & a_4
\end{pmatrix}, \qquad
A_1=
\begin{pmatrix}
b_1 & b_2 \\
b_3 & b_4
\end{pmatrix}.
\]

We will write
\[
T_{i_1 \dots i_d}:=\Tr(A_{i_1}\dots A_{i_d}) \in {\mathbb C}[a_1,\dots,b_4]_d \qquad \text{and} \qquad  W_w := \langle T_b\colon b\in B_w(2,d)\rangle.
\]
By Observation~\ref{obs:dimLinSpan}, we have
\[
\dim \langle \uMPS(2,2,d)\rangle_w= \dim W_w.
\]
The proof of the cases $w=0,1,2,3$ consists in determining that a basis of the vector space~$W_w$ is given by a number of elements that coincides with our conjectured dimension. In order to do this, we guess a set of generators $\{T_b\}$ of the cardinality prescribed by the conjectured dimension. The guess of the $T_b$'s is based on determining a specific \textit{subset} of the set of binary bracelets of weight~$w$, of the proper cardinality. Then, we prove that the generators are linearly independent for a fixed choice of matrices~$A_0$ and~$A_1$.%

The cases $w=0$ and $w=1$ are easy:
\begin{Proposition}
	The space $W_0$ is a $1$-dimensional vector space generated by the polynomial $T_{0,\dots,0}=\Tr\big(A_{0}^d\big)$.
	The space $W_1$ is a $1$-dimensional vector space generated by the polynomial $T_{10\dots 0}=\Tr\big(A_1A_{0}^{d-1}\big)$.
\end{Proposition}
\begin{proof}
	If $w=0$ then $b=({0\dots 0})\in B_0(2,d)$ is the only binary bracelet of weight zero, and
	if $w=1$ then $b=({10\dots 0})\in B_1(2,d)$ is the only binary bracelet of weight $1$.
\end{proof}
We now turn to the case $w=2$. Then Conjecture~\ref{conj: dim} states that $\dim W_w=\big\lfloor\frac{d}{2}\big\rfloor$. But $\big\lfloor \frac{d}{2}\big\rfloor$ is exactly the number $B_2(2,d)$ of generators $T_b$ of $W_w$; hence we need to show that they are linearly independent.
\begin{Proposition}
	The polynomials $\{T_b \colon b\in B_2(2,d)
	\}$ are linearly independent.
\end{Proposition}
\begin{proof}
	Note that
\[
	W_2= \left\langle T_{10^i10^{d-2-i}} \mid i=0,\dots, \left\lfloor \frac{d}{2}\right\rfloor-1 \right\rangle.
\]
	If we make the substitutions
	\[
	A_1=
	\begin{pmatrix}
	0 & 1 \\
	1 & 0
	\end{pmatrix}, \quad
	A_0=
	\begin{pmatrix}
	1 & 0 \\
	0 & x
	\end{pmatrix},
	\]
	our generators $T_b$ become
	\begin{equation}\label{ind polys}
			T_{10^i10^{d-2-i}}=\Tr\big(A_1A_0^iA_1A_0^{d-2-i}\big)= x^i + x^{d-2-i}, \qquad i=0,\dots,\left\lfloor \frac{d}{2}\right\rfloor-1.
	\end{equation}
	Since for the given choice of $A_0$, $A_1$ the polynomials \eqref{ind polys} are $\big\lfloor \frac{d}{2}\big\rfloor$ linearly independent polynomials, the same holds for a generic choice of matrices.
\end{proof}

Finally, we prove the case $w=3$. In this case our conjectured formula states that $\dim W_w = d-3$. Consider the following subset of $B_3(2,d)$:
\begin{equation*}
		\tilde{B}_3:=\{b\in B_3(2,d)\colon b \text{ contains } 11 \text{ {or} } 101\}\subset B_3(2,d).
\end{equation*}
\begin{Lemma}
The cardinality of $\tilde{B}_3$ equals $d-3$.
\end{Lemma}
\begin{proof}
 The cardinality of $\tilde{B}_3$ is the sum of the number of binary bracelets of weight $3$ containing~$11$ and the number of binary bracelets of weight $3$ containing $101$ but not $11$, that are $\big\lceil \frac{d-2}{2}\big\rceil$ and $\big\lceil \frac{d-5}{2}\big\rceil$ respectively.
Therefore the cardinality of $\tilde{B}_3$ is
\[
\left\lceil \frac{d-2}{2}\right\rceil+\left\lceil \frac{d-5}{2}\right\rceil=d-3. \tag*{\qed}
\]	\renewcommand{\qed}{}
\end{proof}

In order to prove the case $w=3$ we need to show that $\big\{T_b\colon b\in \tilde{B}_3\big\}$ is a basis of $W_3$. We first show linear independence:
\begin{Lemma}\label{lemma:w3Indep}
The polynomials $\big\{T_b\colon b\in \tilde{B}_3\big\}$ are linearly independent.
\end{Lemma}

\begin{proof}
We will show that the polynomials are linearly independent even after the following substitution:
\[
A_1=
\begin{pmatrix}
0 & 1 \\
1 & 1
\end{pmatrix}, \qquad
A_0=
\begin{pmatrix}
1 & 0 \\
0 & x
\end{pmatrix}.
\]
Then $W_3$ is spanned by the following polynomials:
\begin{gather*}
f_b := T_{110^b10^{d-b-3}}=x^{b}+x^{d-b-3} + 2x^{d-3}, \qquad b\in\left\{0,\dots,\left\lfloor \frac{d-3}{2}\right\rfloor\right\},\\
g_b := T_{1010^b10^{d-b-4}}=x^{b+1}+x^{d-b-3}+ x^{d-4} +x^{d-3},  \qquad b\in \left\{1,\dots,\left\lfloor \frac{d-4}{2}\right\rfloor\right\}.
\end{gather*}

We now simply have to put the coefficients of these polynomials in a matrix and show it has full rank.
For $d$ even the matrix of coefficients is given by
\[S=
\begin{pmatrix}
1      &       &   &   & \dots&&         &   &  & 3\\
      & 1     &   &        & &&         &  & 1 & 2\\
	   &       & 1 &        & &&         & 1 &  & 2\\
\vdots &       &   & \ddots & && \iddots &   &   & \vdots \\
0	   &       &   &        &1&1&         &  &   & 2 \\
0 & 0 &1      & 0 &   &  		 \dots		&      	&   &   2& 1\\
  &  &0  & 1 &  &        	 			&      	& 1 & 1 &1\\
\vdots  & &       &   & \ddots &   	    		&  	 \iddots 	&  & 	   & \vdots\\
0 & &       & \dots   &   &   	 	 2 		&     \dots    &   0  & 1 & 1
\end{pmatrix}
\]
and for $d$ odd, given by
\[S=
	\begin{pmatrix}
1 &   & & & & \dots & & & & 0 & 3\\
  & 1 &   &  & &      &  &        & 0 & 1 & 2\\
	   &       & 1 &&        & &&         & 1 & 0 & 2\\
\vdots &       &   & \ddots& & && \iddots &   &  \vdots & \vdots \\
       &       &   &       & 1&&1&&& 0&2 \\
	   &       &   &        &&2&         &   &  &0 & 2 \\
0&0 &1      & 0 &   &  		 &\dots	 &      	&  0&   2& 1\\
  & &0      & 1 & 0 &       &			&      	& 1 & 1 &1\\
  & &       &   & \ddots &  &   		&  	 \iddots 	&  & 	  \vdots & \vdots\\
 0&0 &       &   &     	 	& 1&1 		&         &     0& 1& 1
\end{pmatrix}.
\]
By elementary row operations, we can reduce the left upper part to a diagonal matrix of order~$\big\lfloor \frac{d-1}{2}\big\rfloor$. The left lower part is filled with zeros. The (rectangular) right lower block of dimension $\big\lfloor \frac{d-4}{2}\big\rfloor\times \big\lfloor \frac{d-2}{2}\big\rfloor$ can be put in the following upper triangular forms, for $d$ even and odd respectively,
 \begin{gather*}
 \begin{pmatrix}
0 	&	& \dots  & -1 & 2 & -1 \\
\vdots& 	&     -1   & 1 & 1 & -1 \\
&  	\iddots	& \iddots& 0 & 1 & -1 \\
 -1&  1  	&		 &\vdots&\vdots&\vdots\\
 2 &0& \dots &0&1&1
 \end{pmatrix} \to
 \begin{pmatrix}
 2 & 0 & \dots &&& 1 & 1\\
 0 & 2 & &&&3&-1\\
   &  & \ddots&&&5&-3\\
 \vdots & &&&&\vdots&\vdots\\
 0 &\dots&&0&2&*& *\\
 0 &\dots&&&0&*& *
 \end{pmatrix}\quad \text{for }d \text{ even},
\\
 \begin{pmatrix}
0 	&	& \dots  & -1 & 2 & -1 \\
\vdots& 	&       -1 & 1 & 1 & -1 \\
&  	\iddots	& \iddots& 0 & 1 & -1 \\
&    	&		 &\vdots&\vdots&\vdots\\
-1& 1 & &&&-1\\
1 &0& \dots &0&1&0
 \end{pmatrix} \to
\begin{pmatrix}
1       & 0 &\dots&     &0 &1 &0 \\
0       & 1 & 0       &\dots&0 &2 &-1\\
\vdots  & 0 &\ddots   &&&3&-2\\
 &  &&&&\vdots&\vdots\\
 &&&0&1&*&*\\
 &&&&0&*&*
 \end{pmatrix}\quad \text{for }d\text{ odd}.
 \end{gather*}
Both the obtained blocks have rank $\big\lfloor \frac{d-4}{2}\big\rfloor$. We have that the rank of $S$ is $\big\lfloor \frac{d-1}{2}\big\rfloor+\big\lfloor \frac{d-4}{2}\big\rfloor=d-3$, and this concludes the proof.
\end{proof}

We finish our proof by showing that $\big\{T_b\colon b\in \tilde{B}_3\big\}$ spans $W_3$:
\begin{Lemma}\label{lemma:w3Spanning}
	Every polynomial
	 $T_{10^a10^b10^c}=\Tr\big(A_1A_0^aA_1A_0^bA_1A_0^c\big)$, with $1< a\leq b\leq c$, $a+b+c=d-3$ is an element of the linear span $\big\langle T_b\colon b\in \tilde{B}_3\big\rangle$.
\end{Lemma}

\begin{proof}
Notice that the elements of $B_3(2,n)\setminus \tilde{B}_3$ can be written without loss of generality in the form
\[
10^a10^b10^c, \qquad \text{with }1< a\leq b\leq c.
\]
We use induction on $a$. If $a=0$ and $a=1$ then $\big(10^a10^b10^c\big)\in \tilde{B}_3$ and we are done.
If we substitute $A_1 \to A_1A_0^{a-1}$, $A_2 \to A_1A_0^{b}$ and $A_3 \to A_1$
in the equation given by Theorem~\ref{thm:linRel}, we get
\begin{gather*}
\Tr\big(A_1A_0^{a-1} A_1A_0^{b+1} A_1 A_0^c\big) +\Tr\big(A_1A_0^{b} A_1 A_0 A_1A_0^{a+c-1} \big)+ \Tr\big(A_1^2A_0^{a} A_1A_0^{b+c}\big)   \nonumber\\
\qquad{}=\Tr\big(A_1A_0^{a} A_1A_0^{b} A_1 A_0^c\big) + \Tr\big(A_1A_0^{b+1} A_1^2 A_0^{a+c-1} \big) + \Tr\big(A_1 A_0 A_1A_0^{a-1} A_1A_0^{b+c} \big).
\end{gather*}
Reordering the summands, we obtain
\begin{gather*}
T_{10^{a} 10^{b} 1 0^c} = \big(T_{10^{b} 1 0 10^{a+c-1} }+ T_{110^{a} 10^{b+c}} - T_{10^{b+1} 11 0^{a+c-1} } - T_{1 0 10^{a-1} 10^{b+c} }\big)+ T_{10^{a-1} 10^{b+1} 1 0^c}.
\end{gather*}
All terms in the parenthesis have as subscript an elements of $\tilde{B}_3$, and the last term is in $\big\langle\big\{T_b\colon b\in \tilde{B}_3\big\}\big\rangle$ by the induction hypothesis.
This concludes the proof.
\end{proof}

\subsection*{Acknowledgements}
The third author would like to thank Alessandra Bernardi, Jaros{\l}aw Buczy{\'n}ski, Joseph Landsberg, and Frank Verstraete for many helpful discussions. The second author is supported by FWO grants (G023721N and G0F5921N) and UGent BOF grants (BOF21/DOC/182 and STA/201909/038).

\pdfbookmark[1]{References}{ref}
\LastPageEnding

 \end{document}